\documentclass[a4paper,12pt,reqno]{amsart}
\usepackage{vmargin}
\usepackage[pdftex]{graphicx}
\usepackage{amsfonts}
\usepackage{amssymb}
\usepackage[sorted]{amsrefs} 
\usepackage{mathrsfs}
\usepackage{stmaryrd}
\usepackage{amsmath}
\usepackage{amsthm}
\usepackage[shortlabels]{enumitem}
\usepackage{nomencl}
\usepackage[bookmarks=true]{hyperref}   
\usepackage{esint}
\usepackage{dsfont}
\usepackage{upgreek}
\usepackage{xcolor}
\usepackage{slashed}
\usepackage{scalerel}
\usepackage{comment}
\usepackage{todonotes}
\usepackage[utf8]{inputenc}
\usepackage{stackengine}
\usepackage{graphics}
\usepackage{tikz}
\usepackage{pgfplots}

\usepackage{color} 
      
        \definecolor{pink}{rgb}{1,0,1}

        \definecolor{purple}{rgb}{0.4,0.2,1}

\hypersetup{
    colorlinks,%
}

\author{Robert Fulsche}
\author{Medet Nursultanov}
\title[The Sturm-Liouville operator with measure potentials]{Spectral Theory for Sturm-Liouville operators with measure potentials through Otelbaev's function}
\date{\today}
\address{Robert Fulsche, 
Institut f\"ur Analysis,
Leibniz Universit\"at Hannover, 
Welfengarten 1, 30167 Hannover, Germany
}

\email{\href{mailto:fulsche@math.uni-hannover.de }{fulsche@math.uni-hannover.de}}

\address{Medet Nursultanov, School of Mathematics and Statistics,
University of Sydney, Sydney, Australia}
\email{\href{mailto:medet.nursultanov@gmail.com}{medet.nursultanov@gmail.com}}

\thanks{The second author is partially supported by ARC DP190103302 and ARC DP190103451 during this work.}

\def\colour{\colour}

\mathchardef\semic="303B

\newcommand{\barint}{\mbox{$ave \int$}}

\def\barint_#1{\mathchoice
            {\mathop{\vrule width 6pt
height 3 pt depth -2.5pt
                    \kern -8.8pt
\intop}\nolimits_{#1}}%
            {\mathop{\vrule width 5pt height
3 pt depth -2.6pt
                    \kern -6.5pt
\intop}\nolimits_{#1}}%
            {\mathop{\vrule width 5pt height
3 pt depth -2.6pt
                    \kern -6pt
\intop}\nolimits_{#1}}%
            {\mathop{\vrule width 5pt height
3 pt depth -2.6pt
          \kern -6pt \intop}\nolimits_{#1}}}

\definecolor{gr}{rgb}   {0.,   0.8,   0. }
\definecolor{bl}{rgb}   {0.,   0.5,   1. }
\definecolor{mg}{rgb}   {0.7,  0.,    0.7}

\renewcommand{\emptyset}{\varnothing}

 \newtheorem{theorem}{Theorem}[section]
 \newtheorem{lemma}[theorem]{Lemma}
 \newtheorem{corollary}[theorem]{Corollary}
 \newtheorem{proposition}[theorem]{Proposition}

 \theoremstyle{definition}
 \newtheorem{definition}[theorem]{Definition}
 \newtheorem{example}[theorem]{Example}
 
 \newtheorem{remark}[theorem]{Remark}

\begin{document}

\vspace*{-2em}
\maketitle

\begin{abstract}
	We investigate the spectral properties of Sturm-Liouville operators with measure potentials. We obtain two-sided estimates for the spectral distribution function of the eigenvalues. As a corollary, we derive a criterion for the discreteness of the spectrum and a criterion for the membership of the resolvents to Schatten classes. We give two side estimates for the lower bound of the essential spectrum. Our main tool in achieving this is Otelbaev's function.
\end{abstract}

\tableofcontents
\section{Introduction}
Spectral properties of Schr\"{o}dinger and Sturm-Liouville operators have been studied for more
than a century due to numerous applications in physics. There are many publications in this field; we mention only the books \cite{Amrein_Hinz_Pearson, Berezin_Shubin, Cycon_Froese_Kirsch_Simon, T1962, N1968, LS1975, KS1979, LS1991}. Usually, the potential is given by a function with a certain regularity. However, one of the interesting case, which is studied comparably less, is when the potential is given by a measure, $\mu$. In this case, the Sturm-Liouville operator is the operator on $L^2(\mathbb R)$, which can formally be written as the expression
\begin{eqnarray*}
	H_\mu := -\frac{d^2}{dx^2} + \mu.
\end{eqnarray*}
Such operators have been introduced in the physical literature for the description of singular interactions, for instance, point interactions (when $\mu$ is a sum of point measures), and in problems of solid-state physics, nuclear physics, and electromagnetism. There are several approaches to define the Sturm-Liouville operators with distributional potentials, we mention the work \cite{SavchukShkalikov1999}, where the authors define such operators and describe their domains. The way we will take to studying such operators is through quadratic forms.

Sturm-Liouville operators with singular (i.e. measure) potentials became more popular in recent times. In particular, the spectral properties of the Sturm-Liouville operators with point interactions have been studied by various methods. We refer e.g. to \cite{Albeverio_Kostenko_Malamud, Brasche1985, Brasche, Brasche2, M, SavchukShkalikov1999, Brasche_Fulsche, Nursultanov,EKMN} and references therein. 

To the best of the authors' knowledge, Sturm-Liouville operators with measure potentials first appeared in \cite{Zhikov}, where inverse spectral problems for such operators are investigated. The authors of \cite{Brasche1985,Minami} were the first to study self-adjointness and lower semiboundedness of $H_\mu$ for $\mu$ a measure of locally bounded variation.The works \cite{SavchukShkalikov1999,SavchukShkalikov2003,SavchukShkalikov2008} contributed to the theory on finite intervals with potentials in $W^{-1, 2}$. We also mention the papers \cite{HrynivMykytyuk2001,HrynivMykytyuk2002, AlbeverioHrynivMykytyuk2005} for contributions to the spectral theory of Sturm-Liouville operators with singular potentials.

Most of these works consider properties such as self-adjointness, semi-boundedness of the operator, discreteness of the spectrum. However, there is not much information on the distribution of the spectrum. We mention works \cite{Otelbaev, Otelbaev2, R1974, Rozenblum1975} where the authors study the asymptotic behaviour of the eigenvalues of the Schr\"odinger operators with relaxed regularity assumptions on the potentials.

In this work, we study the spectrum of operators $-\frac{d^2}{dx^2} + \mu$, where $\mu$ is a real Radon measure which is ``lower bounded'' in a suitable sense. Our main tool will be the Otelbaev function $q_{\mu_+}^\ast$ associated to the positive part $\mu_+$ of $\mu$. The Otelbaev function first appeared in works by its name giver, Kazakh mathematician Mukhtarbai O. Otelbaev, in the 1970s, when he successfully studied spectral properties of operators $-\frac{d^2}{dx^2} + q$, where $q$ was some lower bounded function \cite{Otelbaev, Otelbaev2}. The main point of using the Otelbaev function (which Otelbaev himself simply denoted as $q^\ast$) was that the method essentially does not need any smoothness assumptions on the potential $q$. Therefore, it is only logical to check if the method works for even more singular potentials, i.e. measures. This is indeed the case, as we shall see in this work. 

We denote by $N(\lambda, A)$ the spectral distribution function of the self-adjoint lower semibounded operator $A$. Upon implementing the method of Otelbaev's function into the setting of measure potentials, we obtain the following result as our main theorem (under reasonable assumptions on the potential $\mu$, which we shall explain later): For any $\lambda \geq 0$ the following estimates hold true:
\begin{equation}\label{main}
M(c_0\lambda+\gamma_0)\leq N(\lambda, H_{\mu}) \leq M(c_1\lambda+\gamma_1),
\end{equation}
where
\begin{equation*}
M(\lambda):=\sqrt{\lambda} ~ \mathcal{L}\left( \{x\in \mathbb{R}: \; q^*_{\mu_+}(x)\leq \lambda\}\right),
\end{equation*}
$\mathcal{L}$ is the Lebesgue measure of the set, $c_0,c_1 \geq 0$ are known constants, and $\gamma_0,\gamma_1 \geq 0$ are constants related to the boundedness of the negative part of the measure $\mu$, which are found more explicitly. These estimates have several applications. By investigating the asymptotic behaviour of $N(\lambda,H_{\mu})$, we derive a criterion for the discreteness of the spectrum of $H_{\mu}$ as well as a Molchanov type criterion. We note that the Molchanov-type criterion is not new: It was proved in \cite{Albeverio_Kostenko_Malamud}, and later improved in \cite{M}, where potentials satisfying Brinck's condition were considered. We also note that sufficient conditions for discreteness of non-semibounded Hamiltonians $H_\mu$ with discrete measure $\mu$ were obtained in \cite{KM2010}, and \cite{IK2010} (see also the survey \cite{KM2013}).

As an another application, we obtain a criterion for the membership of the resolvents of $H_\mu$ in Schatten-von Neumann classes.

We emphasize that \eqref{main} holds for all $\lambda\geq 0$, which allow us, additionally, to obtain several estimates for the eigenvalues. We will present these consequences at the end of our work.

This paper is structured as follows: Section 2 is the technical foundation of this work. Here, we will give precise definitions of our class of measure potentials, discuss the quadratic forms associated to the operators $H_\mu$ and recall several facts from spectral theory. Section 3 serves for the introduction of the Otelbaev function and a discussion of some of its properties. In Section 4, we derive our main theorem, the estimates for the spectral distribution function. Section 5 will present the above-mentioned applications of this theorem. Finally, in Section 6 we will discuss several examples. 

\section{Preliminaries}

\subsection{Notations}
Throughout this paper, $\mathcal{L}$ denotes the Lebesgue measure on $\mathbb R$. We let $\mathrm{D}(\cdot)$ denote the domain of either an operator or a form. By $\sigma(\cdot)$ and $\sigma_{ess}(\cdot)$, we denote spectrum and essential spectrum of an operator, respectively. By $\delta_x$, we denote the Dirac measure concentrated at $x \in \mathbb R$. For $x \in \mathbb R$, by $[x]$ we mean the largest integer not exceeding $x$.
\subsection{Classes of potentials}
Recall that a (positive) measure $\mu$ on the Borel-$\sigma$-algebra $\mathcal B(\mathbb R)$ of $\mathbb R$ is a Radon measure if it satisfies the following properties of local finiteness, outer regularity and inner regularity:
\begin{align*}
\mu(K) &< \infty \text{ for all } K \subset \mathbb R \text{ compact,}\\
\mu(K) &= \inf \{ \mu(O); K \subset O, ~ O \text{ open}\} \text{ for all } K \subset \mathbb R \text{ compact,}\\
\mu(O) &= \sup \{ \mu(K); K \subset O, ~ K \text{ compact}\} \text{ for all } O \subset \mathbb R \text{ open.}
\end{align*}

We will write $\mathfrak P$ for the class of all such positive Radon measures. 
The set of potentials we are interested in is, formally, given by differences of such Radon measures. More precisely, for $\beta\geq0$, we define
\begin{equation*}
\mathfrak{M}_{\beta} = \left\{ (\mu_{+},\mu_{-}) \in   \mathfrak P\times \mathfrak P: \; \mu_{+}(\mathbb{R})>0 \text{ and } \sup_{x\in \mathbb{R}}\mu_{-}([x,x+1]) \leq \beta \right\},
\end{equation*}
and, formally, write 
\begin{equation*}
\mu = \mu_{+} - \mu_{-}, \quad \mu \in \mathfrak{M}_{\beta}.
\end{equation*}
Note that $\mu$ is always a well-defined signed measure on bounded measurable subsets of $\mathbb R$. Nevertheless, if both $\mu_{+}$, $\mu_{-}$ are infinite, the resulting $\mu$ might not be well-defined as a measure on $\mathbb R$:
\begin{example}
	Let
	\begin{equation*}
	\mu_+ = \sum_{k \in \mathbb Z} \delta_k, \quad \mu_- = \sum_{k \in \mathbb Z} \delta_{k + 1/2}.
	\end{equation*}
	Then, $\mu = \mu_+ - \mu_-$ is not a well-defined measure on $\mathcal B(\mathbb R)$. For example, $\mu(\mathbb R) = \infty - \infty$ makes no sense.
\end{example}
However, if one of the two measures $\mu_+, \mu_-$ is finite, then the difference, $\mu = \mu_{+} - \mu_{-}$, is a (signed) measure. In case $\mu_{-} = 0$, we identify $\mu$ with $\mu_{+}$, and simply write $\mu \in \mathfrak P$.

For any $\mu \in \mathfrak {M}_{\beta}$, we always choose $l > 0$ and $\alpha \geq 0$ such that the following hold true: For each interval $I$ of length $\geq l$ we have $\mu_-(I) \geq \alpha$. At first sight, this seems to be a bit odd, since this is satisfied for any $l > 0$ and $\alpha = 0$. The point is the following: If we can choose $l > 0$ and $\alpha > 0$, which is not always possible, then this will yield additional spectral information. In what follows, given $\mu \in \mathfrak {M}_{\beta}$ we will always denote by $\alpha, l$ the above constants. 
\begin{example} 
		Let $\mu_{+}= \mathcal{L}$, $0 < a < b < \infty$, and $\{\alpha_k\}_{k \in \mathbb Z} \subset [a, b]$. Further, let $d > 0$. Then, for
		\begin{align*}
		\mu_- = \sum_{k \in \mathbb Z} \alpha_k \delta_{dk}
		\end{align*}
		we can choose the constants $l > d$, $\alpha \leq a$ and $\beta \geq b( [1/d] + 1)$.
\end{example}	
		
\begin{example}
    If $f: \mathbb R \to [a, b]$, where $0 < a < b$, then for
	\begin{align*}
	\mathrm {d}\mu_-(x) = f(x) ~\mathrm{d}x
	\end{align*}
	we can choose $l > 0$, $\alpha \leq la$ and $\beta \geq b$.
\end{example}

\subsection{Formulation of the problem}
Let $\mu \in \mathfrak {M}_{\beta}$. We want to investigate the spectral properties of the operator $H_{\mu}$ associated to the quadratic form
\begin{equation*}
a_{\mu}(f,g)=\int_{\mathbb{R}}f'(x)\overline{g'(x)}~\mathrm dx+\int_{\mathbb{R}}f\overline{g}~\mathrm d\mu_+-\int_{\mathbb{R}}f\overline{g}~\mathrm d\mu_-
\end{equation*}
with domain
\begin{equation*}
\mathrm{D}(a_{\mu})=\left\{f\in H^1(\mathbb{R}): \lim_{\substack{x \to \infty\\ y \to -\infty}} \int_{y}^x |f|^2 ~\mathrm d\mu \text{ exists} \right\}
\end{equation*}
in the sense of Kato's first representation theorem \cite[Theorem 2.6]{Kato}. This is justified by the following fact.
\begin{proposition}\label{form}
Let $\beta\geq 0$ and $\mu \in \mathfrak {M}_{\beta}$, then the form $a_\mu$ is lower semibounded with bound $-2\max\{\beta, \beta^2\}$, closed, symmetric and densely defined.
\end{proposition}
We note that Proposition \ref{form} is not new: A stronger result is obtained in \cite{M}, see also \cite{Albeverio_Kostenko_Malamud} for related results. In principle, it is a simple extension of Brasche's results \cite{Brasche1985}, who seemingly was the first to study properties of Schr\"{o}dinger operators with general measure potential, such as self-adjointness and lower semiboundedness, through form methods. Indeed, if $\mu \in \mathfrak P$, the proposition is exactly \cite[Theorem 1]{Brasche1985}, while the full statement of the proposition is very close to, but not immediately implied by \cite[Theorem 3]{Brasche1985}. For the reader's convenience, we briefly demonstrate how to extend \cite[Theorem 1]{Brasche1985} to obtain this statement. The proof hinges on the following well-known fact:
\begin{lemma}[{\cite[Lemma 2]{Brinck}}]\label{lemma_Brinck}
Let $I$ be an interval of length $d$ and $f \in H^1(I)$. Then, for all $y \in I$ the following estimate holds true:
\[ \frac{1}{2d} \int_{I}|f(x)|^2 ~\mathrm dx - \frac{d}{2}\int_{I}|f'(x)|^2 ~\mathrm dx \leq |f(y)|^2 \leq \frac{2}{d} \int_{I}|f(x)|^2 ~\mathrm dx + d \int_{I}|f'(x)|^2 ~\mathrm dx.\]
\end{lemma}

\begin{proof}[Proof of Proposition \ref{form}]
Symmetry is obvious. Further, $C_c^\infty(\mathbb R) \subset \mathrm D(a_\mu)$, hence the form is densely defined. If $\mu \in \mathfrak P$, then $a_\mu$ is obviously positive and, by \cite[Theorem 1]{Brasche1985} also closed with
\begin{equation*}
\mathrm D(a_\mu) = \{ f \in H^1(\mathbb R): ~f \in L^2(\mathbb R, ~\mu)\}. 
\end{equation*}

Let $\beta> 0$ and $\mu \in \mathfrak {M}_{\beta}$. Let $d < \min\{ 1, 1/\beta\}$ and decompose $\mathbb R$ into the intervals $I_k = ((k-1)d, kd]$ for $k \in \mathbb Z$. Then, for every $f \in \mathrm D(a_{\mu_+})$ we have by the previous lemma:
\begin{align*}
\left | \int_{\mathbb R} |f|^2 ~\mathrm d\mu_- \right| &\leq \sum_{k \in \mathbb Z } \left | \int_{I_k} |f|^2 d\mu_{-} \right |\\ 
&\leq \beta \sum_{k \in \mathbb Z} \sup_{x\in I_k}|f(x)|^2\\
&\leq \beta \sum_{k \in \mathbb Z} \left[ \frac{2}{d} \int_{I_k} |f(x)|^2 ~\mathrm dx + d\int_{I_k} |f'(x)|^2 ~\mathrm dx \right]\\
&= \frac{2\beta}{d} \int_{\mathbb R} |f(x)|^2 ~\mathrm dx + \beta d \int_{\mathbb R} |f'(x)|~\mathrm dx\\
&\leq \frac{2\beta}{d} \int_{\mathbb R}|f(x)|^2 ~\mathrm dx + \beta d \left[ \int_{\mathbb R} |f'(x)|^2 ~\mathrm dx + \int_{\mathbb R} |f|^2 ~\mathrm d\mu_{+} \right]
\end{align*}
Since $\beta d < 1$, the KLMN-Theorem \cite[Theorem X.17]{Reed2} implies that $a_\mu$ is closed, lower semibounded with bound $-\frac{2\beta}{d}$ and $\mathrm D(a_\mu) = \mathrm D(a_{\mu_+})$. Since the integral $\int_{\mathbb R} |f|^2 ~\mathrm d\mu_-$ exists for every $f \in H^1(\mathbb R)$ by the above estimates, we also obtain
\begin{align*}
\mathrm D(a_\mu) = \mathrm D(a_{\mu_+}) &= \{ f \in H^1(\mathbb R): ~f \in L^2(\mathbb R, \mu_+)\}\\
 &= \left\{f\in H^1(\mathbb{R}): \lim_{\substack{x \to \infty\\ y \to -\infty}} \int_{y}^x |f|^2 ~\mathrm d\mu \text{ exists} \right\}.
\end{align*}
\end{proof}

\begin{remark}
As we just saw, we obtain the lower bound $-2\max\{\beta, \beta^2\}$. It seems strange that expression for the lower bound might be different depending on $\beta \geq 1$ or $\beta < 1$. We will see in Theorem \ref{estimate_smallest_eigenvalue} that $-3\beta$ is always a lower bound. We note that semi-boundedness of the operator is essential for the method presented here. There are some methods for non-semibounded operators, see the review in \cite{NR2017}.
\end{remark}

\subsection{Auxiliary lemmas}
Here, we will state and prove some auxiliary results. We begin with the following definition.
\begin{definition}
	The spectral distribution function $N(\lambda, A)$ of the self-adjoint operator A is defined as
	\[ N(\lambda, A) = \begin{cases} \sum_{\nu < \lambda} \dim Eig(A, \nu),~ &\sigma_{ess}(A) \cap (-\infty, \lambda) = \emptyset\\
	\infty,& \text{otherwise} \end{cases} \]
	where $Eig(A, \nu)$ is the eigenspace of $A$ w.r.t. the eigenvalue $\nu$ (by definition $\{ 0\}$ if $\nu$ is not an eigenvalue).
\end{definition}
The following lemma is well-known in spectral theory.  We state it here for the reader's convenience.
\begin{lemma}[Glazman Lemma]\label{Glazman}
	Let $A$ be a lower-semibounded self-adjoint operator in a Hilbert space $\mathcal{H}$ with corresponding closed sesquilinear form $a$ and form domain $\mathrm D(a)$. Then, it holds
	\begin{equation*}
	N(\lambda,A)=\sup \{\mathrm{dim} L;~ L \text{ subspace of } \mathrm D(a) \text{ s.th. } a(u,u)<\lambda\langle u,u\rangle \text{ for } u\in L \setminus \{ 0\} \}.
	\end{equation*}
\end{lemma}
We present one last lemma in this section. The result is neither new nor surprising, being a simple consequence of the standard domain-bracketing. Again, we add a proof for completeness.
\begin{lemma}\label{ineq_for_N}
	Let $d>0$ and for each $k \in \mathbb Z$ set $I_k=((k-1)d,kd]$. For $\mu \in \mathfrak M$ consider the operators $H_{\mu}^k$ associated to the sesquilinear forms
	\begin{equation*}
	a_{\mu}^k(f,g)=\int_{I_k}f'(x)\overline{g'(x)}~\mathrm dx+\int_{I_k}f\overline{g}~\mathrm d\mu
	\end{equation*}
	with domain $\mathrm{D}(a_{\mu}^k)= H^1(I_k)$. Then, the following holds true for every $\lambda \in \mathbb R$:
	\begin{equation*}
	N(\lambda,H_{\mu})\leq \sum_{k=-\infty}^{\infty}N(\lambda,H_{\mu}^k).
	\end{equation*}
\end{lemma}

\begin{proof}
	It holds
	\[ H^1(\mathbb R) \overset{\iota}{\hookrightarrow}  \bigoplus_{k \in \mathbb Z} H^1(I_k), ~ \iota(f) = (f|_{I_k})_{k \in \mathbb{Z}}. \] 
	and hence
	\[ \mathrm D(a_{\mu}) \overset{\iota}{\hookrightarrow} \bigoplus_{k \in \mathbb Z} \mathrm D(a_{\mu}^k) = \bigoplus_{k \in \mathbb Z} H^1(I_k) \]
	injectively. Setting
	\[ A_1 := \{ L \subset \iota(\mathrm D(a_{\mu})); ~ \sum_{k \in \mathbb Z} a_{\mu}^k(f_k, f_k) < \lambda ~ \text{ for } ~ (f_k)_{k \in \mathbb Z} \in L ~ \text{ with } \| f\|_{\oplus_k L^2(I_k)} = 1\} \]
	and 
	\[ A_2 := \{ L \subset \bigoplus_{k \in \mathbb Z} H^1(I_k); ~ \sum_{k \in \mathbb Z} a_{\mu}^k(f_k, f_k) < \lambda ~ \text{ for } ~ (f_k)_{k \in \mathbb Z} \in L ~ \text{ with } \| f\|_{\oplus_k L^2(I_k)} = 1\}, \]
	we trivially have $A_1 \subseteq A_2$. For $f \in \mathrm D(a_{\mu})$ we have
	\[ a_{\mu}(f,f) = \sum_{k \in \mathbb Z} a_{\mu}^k(f|_{I_k}, f|_{I_k}). \]
	Using that 
	\[ L^2(\mathbb R) \cong \bigoplus_{k \in \mathbb Z} L^2(I_k) \]
	isometrically, we have by Glazman's lemma
	\begin{align*}
	N(\lambda, H_{\mu}) = \sup_{L \in A_1} \dim L \leq \sup_{L \in A_2} \dim L = N(\lambda, \bigoplus_{k \in \mathbb Z} H_{\mu}^k)
	\end{align*}
	The simple fact that
	\[ N(\lambda, \bigoplus_{k \in \mathbb Z} H_{\mu}^k) = \sum_{k \in \mathbb Z} N(\lambda, H_{\mu}^k) \]
	finishes the proof.
\end{proof}

\section{Otelbaev's function}
In this section we introduce Otelbaev's function and discuss its properties.
\begin{definition}
	For $\mu \in \mathfrak P$ we set 
	\begin{equation*}
	d_\mu(x) := \sup \{d \geq 0: d\mu([x-d/2, x+d/2]) \leq 1 \}.	
	\end{equation*}
	We then define the Otelbaev function $q_\mu^\ast$ on $\mathbb R$ as
	\begin{equation*}
	q_{\mu}^*(x):= 1/d_\mu(x)^2.
	\end{equation*}
\end{definition}

\begin{remark}
If $\mu \in \mathfrak P$ is a continuous measure, i.e. its cumulative distribution function
\begin{equation*}
F_\mu(x) := \begin{cases}
\mu((0, x]), \quad &x \geq 0\\
-\mu((x, 0]), \quad &x < 0
\end{cases}
\end{equation*}
is continuous, then $d_\mu(x)$ is the unique solution to the equation
\begin{equation*}
d \cdot \mu\left(\left [x- \frac{d}{2}, x + \frac{d}{2} \right ] \right) = 1.
\end{equation*}
\end{remark}

\begin{proposition}\label{prop_of_q}
	For any $0 \neq \mu \in \mathfrak P$ the function $q_{\mu}^{*}$ is strictly positive everywhere.
\end{proposition}
\begin{proof}
	We instead show positivity and finiteness of $d_\mu(x)$, which is of course equivalent. By assumption, there is some constant $C_0$, $1 \leq C_0 < \infty$, such that
	\[ \mu([x-1/2, x + 1/2]) \leq C_0. \]
	Hence, for $d = 1/C_0$, $\mu([x-d/2, x+d/2]) \leq 1/d$. Further, for $d$ large enough, $\mu([x-d/2, x+d/2]) > 1/d$ (as $\mu(\mathbb R) > 0$). By monotonicity of $\mu([x-d/2, x+d/2])$ and $1/d$ in $d$ (the first being monotonely increasing, the second decreasing) it needs to hold $0 <d_{\mu}(x)< \infty$.
\end{proof}

\begin{definition}
	Let $\mu_n$, $\mu \in \mathfrak M_\beta$. We say that $\mu_n \to \mu$ in the weak*-sense (write $\mu_n \overset{w^\ast}{\longrightarrow} \mu$), if 
	\[ \int f d\mu_n \to \int fd\mu \]
	for all $f \in C_c(\mathbb R)$.
\end{definition}

\begin{proposition}
	Assume $\mu_n, \mu \in \mathfrak P$ and $\mu_n \overset{w^\ast}{\longrightarrow} \mu$. Then, $q_{\mu_n}^\ast(x) \to q_{\mu}^\ast(x)$ for all $x \in \mathbb R$.
\end{proposition}
\begin{proof}
For $K \subset \mathbb R$ compact and $O \subset \mathbb R$ open and bounded it is an easy exercise to prove the following:
\begin{align}
\limsup_{n \to \infty} \mu_n(K) &\leq \mu(K) \label{measurelimit_compact},\\
\liminf_{n \to \infty} \mu_n(O) &\geq \mu(O). \label{measurelimit_open}
\end{align}
	Fix $x \in \mathbb R$. We will prove
\begin{align}
\liminf_{n \to \infty} d_{\mu_n}(x)  \geq d_\mu(x) \label{liminf},\\
\limsup_{n \to \infty} d_{\mu_n}(x) \leq d_\mu(x), \label{limsup}
\end{align}	
which is of course equivalent to the statement. Let us first discuss (\ref{liminf}).

Let $d$ be such that $0 < d < d_\mu(x)$. Set $\tau = \frac{d_\mu(x) - d}{2} > 0$ and $\varepsilon = \frac{1}{d} - \frac{1}{d + \tau} > 0$. In particular, $d + \tau < d_\mu(x)$. According to (\ref{measurelimit_compact}), there is some $N \in \mathbb N$ such that for $n \geq N$ we have
\begin{align*}
\mu_n([x-d/2, x+d/2]) &\leq \mu([x-d/2, x+d/2]) + \varepsilon.
\end{align*}
Monotonicity of the measure then implies
\begin{align*}
\mu_n([x-d/2, x+d/2]) &\leq \mu([x-(d+\tau)/2, x+(d+\tau)/2]) + \varepsilon\\
&\leq \frac{1}{d+\tau} + \varepsilon = \frac{1}{d}.\\
\end{align*}
In particular, we obtain $d_{\mu_n}(x) \geq d$ for all $n \geq N$, i.e. $\liminf_{n \to \infty} d_{\mu_n}(x) \geq d$. Since $d < d_\mu(x)$ was arbitrary, this yields $\liminf_{n \to \infty} d_{\mu_n}(x) \geq d_\mu(x)$.

The estimate (\ref{limsup}) is proven similarly: Let $d > d_\mu(x)$ be arbitrary and set $\tau = \frac{d - d_\mu(x)}{2} > 0$, $\varepsilon = \frac{1}{d-\tau} - \frac{1}{d} > 0$. In particular, $d - \tau > d_\mu(x)$. According to (\ref{measurelimit_open}) we have
\begin{equation*}
\mu_n((x-d/2, x+d/2)) \geq \mu((x - d/2, x+d/2)) - \varepsilon
\end{equation*}
for $n$ large enough. For such $n$, we obtain
\begin{align*}
\mu_n([x-d/2, x+d/2]) &\geq \mu_n((x-d/2, x+d/2))\\
&\geq \mu((x-d/2, x+d/2)) - \varepsilon\\
&\geq \mu([x-(d-\tau)/2, x+(d-\tau)/2]) - \varepsilon\\
&\geq \frac{1}{d-\tau} - \varepsilon = \frac{1}{d}.
\end{align*}
This proves $d_{\mu_n}(x) \leq d$ for large values of $n$. Since $d > d_\mu(x)$ was arbitrary, we obtain $\limsup_{n \to \infty} d_{\mu_n}(x) \leq d_\mu(x)$.
\end{proof}

\begin{corollary}\label{q_cont}
	For $\mu \in \mathfrak P$ the Otelbaev function $q_{\mu}^\ast$ is continuous on $\mathbb R$.
\end{corollary}
\begin{proof}
	For $y \in \mathbb R$ let the measure $\mu_y$ be defined by
	\[ \mu_y(A) = \mu(A - y). \]
	By the dominated convergence theorem, $\mu_y \overset{w^\ast}{\to} \mu$ as $y \to 0$, hence
	\[ q_{\mu_y}^{\ast}(x) = q_{\mu}^{\ast}(x-y) \to q_{\mu}^{\ast}(x). \]
\end{proof}

\begin{example}
Consider the measure $\mathrm d\mu(x) = |x| \mathrm dx$. For $d > 0$ it holds
		\[ \mu([x-d/2, x+d/2]) = \frac{1}{8} \, {\left(d + 2 \, x\right)} {\left| d + 2 \, x \right|} + \frac{1}{8} \, {\left(d - 2 \, x\right)} {\left| -d + 2 \, x \right|}. \]
		Let $x > \frac{1}{2^{2/3}}$. Then, $d_0(x) = \frac{1}{\sqrt{x}}$ solves the equation
		\[ \mu([x-d/2, x+d/2]) = \frac{1}{d}. \]
		This implies $q_\mu^\ast(x) = x$ for $x > 1/2^{2/3}$. By symmetry, one obtains
		\[ q_\mu^\ast(x) = |x| \quad \text{ for } |x| > 1/2^{2/3}. \]
\end{example}

\begin{example}
Consider the measure $\mathrm d\mu(x) = x^2 \mathrm dx$. For $d > 0$ one can easily compute that
		\[ \mu([x-d/2, x+d/2]) = \frac{d^3}{12} + d\cdot x^2. \]
		From this, one sees that the equation
		\[ \mu([x-d/2, x+d/2]) = 1/d \]
		has the unique solution
		\[ d_\mu(x) = \sqrt{-6x^2 + 2\sqrt{9x^4 + 3}}. \]
		Hence, we obtain
		\[ q_\mu^\ast(x) = \frac{1}{2\sqrt{9x^4 + 3} - 6x^2} = \frac{\sqrt{x^4 + 1/3} + x^2}{2} \approx x^2. \]
		With this concrete formula, it is standard to show that
		\[ q_\mu^\ast(x) - x^2 = \frac{\sqrt{x^4 + 1/3} - x^2}{2} \in O\left( \frac{1}{x^2} \right) \]
		as $x \to \pm \infty$.
\end{example}
\begin{example}
		For $\alpha > 0$ let $\mu = \alpha \delta_0$. Then,
		\[ \alpha \delta_0([x-d/2, x+d/2]) = \begin{cases}
		\alpha, \quad & d\geq 2|x|\\
		0, \quad & d < 2|x|
		\end{cases}, \]
		which implies 
		\[ q_{\alpha \delta_0}^\ast(x) = \begin{cases}
		\alpha^2, \quad & 2|x|\leq \frac{1}{\alpha}\\
		\frac{1}{4|x|^2}, \quad & 2|x| > \frac{1}{\alpha}
		\end{cases}. \]
\end{example}
\begin{example}
        Let $\mu$ be the Cantor measure, i.e. the singular continuous measure which has the ``Devil's staircase'' (cf. Figure \ref{Figure_distribution_cantor}) as its cumulative distribution function. A good numerical approximation to $q_\mu^\ast$ can be seen in Figure \ref{figure_cantor_otelbaev}.
\end{example}
\begin{figure}[!htb]
	\begin{minipage}{0.50\textwidth}
		\hspace*{-40pt}\centering
		\resizebox{0.8\linewidth}{!}{\input{figure1.tex}}
		\caption{The cumulative distribution function of the Cantor measure (``Devil's staircase'')}
		\label{Figure_distribution_cantor}
	\end{minipage}
	\hspace*{-50pt}\begin{minipage}{0.50\textwidth}
		\centering
		\resizebox{1.2\linewidth}{!}{\input{figure2.tex}}
		\caption{The Otelbaev function of the Cantor measure}
		\label{figure_cantor_otelbaev}
	\end{minipage}
\end{figure}

Next we prove the following auxiliary lemma, which we will use later.
\begin{lemma}\label{est_q*}
	Let $\mu \in \mathfrak P$, $d>0$ and $z\in \mathbb{R}$. Assume $\sqrt{q_{\mu}^*(z)} < 1/(2d)$, then
	\begin{eqnarray*}
		\sqrt{q_{\mu}^*(x)}\leq \frac{1}{d}, \quad \text{for }  x\in [z-d/2, z+d/2].
	\end{eqnarray*}
		
\end{lemma}
	\begin{proof}
		By the assumption and the definition of $q_{\mu}^*$, we obtain
		\begin{eqnarray*}
			\frac{1}{2d}\geq \mu([z-d,z+d]).
		\end{eqnarray*}
		Thus, for $x\in [z-d/2, z+d/2]$, we obtain
		\begin{eqnarray*}
			\mu([x-d/2,x+d/2])\leq \mu([z-d,z+d])\leq \frac{1}{2d}\leq \frac{1}{d},
		\end{eqnarray*}
		which finishes the proof.
	\end{proof}

The following lemma, dealing with Otelbaev's function in the absolutely continuous case, is certainly not new, being contained in framework of Otelbaev's original work. Nevertheless, we could not locate this precise statement in the literature. Since it will be useful in the discussion of some examples, we state it with proof.
\begin{lemma}\label{Lemma_convex_concave}
Let $\mathrm d\mu(x) = f(x)\mathrm dx$ with $f \geq 0$ measurable.
\begin{enumerate}
\item If $f$ is continuous, convex and monotonely increasing on $(a,\infty)$ for some $a \in \mathbb R$, then $q_{\mu}^\ast(x) \geq f(x)$ for large values of $x$.
\item If $f$ is continuous, convex and monotonely decreasing on $(-\infty, a)$ for some $a \in \mathbb R$, then $q_{\mu}^\ast(x) \geq f(x)$ for small values of $x$.
\item If $f$ is continuous, concave and monotonely increasing on $(a,\infty)$ for some $a \in \mathbb R$, then $q_{\mu}^\ast(x) \leq f(x)$ for large values of $x$.
\item If $f$ is continuous, concave and monotonely decreasing on $(-\infty, a)$ for some $a \in \mathbb R$, then $q_{\mu}^\ast(x) \leq f(x)$ for small values of $x$.
\end{enumerate}
\end{lemma}

\begin{proof}
We only prove (1), the other statements follow analogously. Assume $f$ is not identically $0$ on $(a, \infty)$. For $\mathrm d\mu(x) = f(x)\mathrm dx$ with continuous, non-negative function $f$, $d_\mu(x)$ is the unique solution to the equation
\begin{equation*}
 d \int_{x-d/2}^{x+d/2} f(\tau) d\tau = 1.
\end{equation*}
In particular, if $d \int_{x-d/2}^{x+d/2}f(\tau) d\tau \geq 1$, then $d \geq d_\mu(x)$. For $x$ large enough such that $f(x) \neq 0$ and $x - \frac{1}{2\sqrt{f(x)}} > a$ we set $d = \frac{1}{\sqrt{f(x)}}$. Then, using the mean value theorem,
\begin{align*}
\frac{1}{\sqrt{f(x)}} \int_{x-\frac{1}{2\sqrt{f(x)}}}^{x+\frac{1}{2\sqrt{f(x)}}} f(\tau) d\tau = \frac{f(\xi)}{f(x)}
\end{align*}
for some $\xi \in \left (x-\frac{1}{2\sqrt{f(x)}}, x + \frac{1}{2\sqrt{f(x)}} \right )$. An easy application of Jensen's inequality and the monotonicity of $f$ implies that we actually have $\xi \geq x$. Therefore, again by the monotonicity,
\begin{align*}
\frac{f(\xi)}{f(x)} \geq 1,
\end{align*}
which proves $d_\mu(x) \leq \frac{1}{\sqrt{f(x)}}$ and therefore $q_{\mu}^\ast(x) \geq f(x)$.
\end{proof}

Finally, we prove the following perturbation result for the Otelbaev function:

\begin{lemma}\label{perturation_lemma}
Let $\mu \in \mathfrak P$ such that $q_\mu^\ast(x) \to \infty$ as $|x| \to \infty$. Further, let $\nu \in \mathfrak P$ be such that $\sup_{x\in \mathbb{R}}\nu([x,x+1]) \leq C_0$ for some finite constant $C_0$. Then, there are constants $c_1, c_2> 0$ such that for $|x|$ large we have
\begin{equation}
c_1 q_{\mu + \nu}^\ast(x) \leq q_\mu^\ast(x) \leq c_2 q_{\mu + \nu}^\ast(x).
\end{equation}
\end{lemma}
\begin{proof}
Let $\nu \neq 0$. We will equivalently prove that
\begin{equation*}
C_1 d_{\mu + \nu}(x) \leq d_{\mu}(x) \leq C_2 d_{\mu + \nu}(x)
\end{equation*}
for $|x|$ large and some constants $C_1, C_2 > 0$. Since $q_{\mu}^\ast(x) \to \infty$ as $|x| \to \infty$, there is some $x_0 > 0$ such that $d_{\mu}(x) \leq \min\{ 1, 1/C_0 \}$ for $|x| \geq x_0$. Then,
\begin{align*}
&\frac{d_\mu(x)}{4}\left( \mu \left( \left[ x- \frac{d_\mu(x)}{8}, x+\frac{d_\mu(x)}{8}\right] \right) + \nu \left( \left[ x - \frac{d_\mu(x)}{8}, x + \frac{d_\mu(x)}{8}\right] \right) \right)\\
&\leq \frac{1}{2}\frac{d_\mu(x)}{2}\mu \left( \left[ x - \frac{d_\mu(x)}{4}, x + \frac{d_\mu(x)}{4}\right] \right) + \frac{d_\mu(x) C}{4}\\
&\leq \frac{1}{2} + \frac{1}{4} \leq 1,
\end{align*}
which gives
\begin{equation*}
\frac{d_\mu(x)}{4} \leq d_{\mu + \nu}(x).
\end{equation*}
By definition,
\begin{align*}
\frac{d_{\mu+\nu}(x)}{2}& \Bigg(\mu \left( \left[x - \frac{d_{\mu+\nu}(x)}{4}, x + \frac{d_{\mu+\nu}(x)}{4}\right]\right)\\
 &\quad + \nu\left(\left[x-\frac{d_{\mu+\nu}(x)}{4}, x + \frac{d_{\mu+\nu}(x)}{4}\right]\right) \Bigg) \\
 &\leq 1.
\end{align*}
This yields
\begin{align*}
&\frac{d_{\mu + \nu}(x)}{2} \mu\left( \left[ x - \frac{d_{\mu + \nu}(x)}{4}, x + \frac{d{\mu + \nu}(x)}{4} \right] \right) \\
&\leq \frac{\mu([x-d_{\mu+\nu}(x)/4, x+d_{\mu + \nu}(x)/4])}{\mu([x-d_{\mu+\nu}(x)/4, x+d_{\mu + \nu}(x)/4]) + \nu([x-d_{\mu+\nu}(x)/4, x+d_{\mu + \nu}(x)/4])}\\
&\leq 1,
\end{align*}
which gives
\begin{equation*}
\frac{d_{\mu + \nu}(x)}{2} \leq d_{\mu}(x).
\end{equation*}
\end{proof}

\section{Two sided estimates of the spectral distribution function}
In this section we obtain two-sided estimates for the spectral distribution function of the operator $H_{\mu}$ in terms of the following function
\begin{equation*}
	M(\lambda):=\sqrt{\lambda} ~ \mathcal{L}\left( \{x\in \mathbb{R}: \; q^*_{\mu_+}(x)\leq \lambda\}\right).
\end{equation*}
We will prove the following theorem:
\begin{theorem}\label{theorem_twosided_estimate}
	Let $\beta\geq 0$, $\mu \in \mathfrak {M}_{\beta}$, and $\alpha, l \geq 0$ be the constants belonging to $\mu_-$ as discussed above. Then, the distribution function of the eigenvalues of $H_{\mu}$ satisfies
	\begin{align*}
		2M\left(\frac{3(\frac{\alpha}{l} + 2\lambda)}{16(\pi^2 + 1)(\alpha l+2)}\right) \leq N(\lambda, H_{\mu})\leq M\left((\sqrt{4\lambda + 9\beta^2 + 12\beta} + 3\beta)^2\right)
	\end{align*}
	for every $\lambda \geq 0$. In particular, if $\mu_{-}=0$, then
	\begin{equation*}
	2M\left(\frac{3\lambda}{16(\pi^2 + 1)}\right) \leq N(\lambda, H_{\mu})\leq M\left(4\lambda\right).
	\end{equation*}
\end{theorem}
We split the proof into two separate parts. 

\subsection{Upper bound for the spectral distribution function}
First we estimate the spectral distribution function from above.
\begin{theorem}\label{theorem_upperbound}
Let $\beta\geq 0$, $\mu \in \mathfrak {M}_{\beta}$, and $\lambda \geq 0$. Then the following inequality holds true:
\begin{align*}
N&(\lambda, H_\mu)\\
&\leq (\sqrt{4\lambda + 9\beta^2 + 12\beta} + 3\beta) \mathcal{L} \left( \left \{ x\in \mathbb R: ~q_{\mu_+}^\ast(x) \leq (\sqrt{4\lambda + 9\beta^2 + 12\beta} + 3\beta)^2 \right \} \right).
\end{align*}
\end{theorem}
\begin{proof}
If $\lambda = \beta = 0$, then $H_{\mu}$ is non-negative, so that the statement of the theorem will be $0\leq 0$, which is true. Assume that $\max\{\lambda,\beta\}>0$. Then, we define
\begin{equation*}
d = \frac{2}{\sqrt{4\lambda + 9\beta^2 + 12\beta} + 3\beta}.
\end{equation*}
Note that $0<d<\infty$,
\begin{equation}\label{lambda_d_relation}
\lambda = \frac{1}{d^2} - \frac{3 \beta}{d} - 3\beta
\end{equation}
and
\begin{align*}
1 - \beta d([ d ] + 1) &\geq 1 - \frac{2\beta}{\sqrt{4\lambda + 9\beta^2 + 12\beta} + 3\beta} \left(\frac{2}{\sqrt{4\lambda + 9\beta^2 + 12\beta} + 3\beta} + 1\right)\\
 &\geq 1 - \frac{2\beta}{\sqrt{ 9\beta^2 + 12\beta} + 3\beta} \left(\frac{2 + \sqrt{ 9\beta^2 + 12\beta} + 3\beta}{\sqrt{9\beta^2 + 12\beta} + 3\beta}\right) = \frac{2}{3} > 0,
\end{align*}
where $[d]$ denotes the largest integer not exceeding $d$. 

For every $k \in \mathbb Z$ we denote by $I_k$ the interval $((k-1)d, kd]$ of length $d$. Let $H_{\mu}^k$ be the operators defined in Lemma \ref*{ineq_for_N}, then
	\begin{equation}
	N(\lambda,H_{\mu})\leq \sum_{k=-\infty}^{\infty}N(\lambda,H_{\mu}^k).
	\end{equation}
	We are going to further estimate the right-hand side of this inequality. For this, we are first going to discard certain values of $k$ which do not contribute to this sum, i.e. values of $k$ for which we have $N(\lambda, H_\mu^k) = 0$. Since $\mu_+$ is a positive measure and $\mu_{-}$ satisfies
	\begin{equation}\label{b type}
	\sup_{x\in \mathbb{R}}\mu_{-}([x,x+1]) \leq \beta
	\end{equation}
	with the constant $\beta \geq 0$, we obtain for $f \in H^1(I_k)$:
	\begin{equation*}
	a_{\mu}^{k}(f,f) \geq \int_{I_k} |f'(x)|^2 ~\mathrm dx -\int_{I_k} |f(x)|^2 ~\mathrm d\mu_{-} \geq \int_{I_k} |f'(x)|^2 ~\mathrm dx - \beta([d] + 1)\sup_{x \in I_k} |f(x)|^2
	\end{equation*}
	Applying Lemma \ref{lemma_Brinck}, we derive
	\begin{equation}\label{lower_bound_for form}
	a_{\mu}^{k}(f,f) \geq (1 - \beta d([d] + 1)) \int_{I_k} |f'(x)|^2 ~\mathrm dx - \frac{2\beta([d] + 1)}{d} \int_{I_k} |f(x)|^2 ~\mathrm dx.
	\end{equation}
	Let $S_{k,\lambda}$ be the set of functions from $D(a_{\mu}^k)$ which satisfy the following inequality:
	\begin{equation}\label{upperbound1}
	(1 - \beta d([d] + 1))\int_{I_k}|f'(x)|^2~\mathrm dx - \frac{2\beta ([d] + 1)}{d} \int_{I_k}|f(x)|^2~\mathrm dx\leq \lambda\int_{I_k}|f(x)|^2~\mathrm dx.
	\end{equation}
	If $f \in \mathrm D(a_\mu^k) \setminus S_{k, \lambda}$, then \eqref{lower_bound_for form} implies that $a_{\mu}(f,f) > \lambda (f,f)_{L_2(I_k)}$. Such functions cannot be eigenfunctions to eigenvalues less than $\lambda$. In particular, we do not need to consider values of $k \in \mathbb Z$ for which we have $S_{k, \lambda} = \{ 0\}$. Assume $\mu_+(I_k) > 2/d$ and let $f \in S_{k, \lambda} \setminus \{ 0\}$. By using Lemma \ref{lemma_Brinck} and \eqref{b type}, we estimate
	\begin{align*}
	a_{\mu}^{k}(f,f) &\geq \int_{I_k} |f'(x)|^2 ~\mathrm dx + \mu_{+}(I_k) \inf_{x \in I_k} |f(x)|^2 - \beta([d] + 1)\sup_{x \in I_k} |f(x)|^2\\
	& \geq \left( 1 - \beta([d] + 1)d -\frac{d}{2}\mu_{+}(I_k) \right) \int_{I_k} |f'(x)|^2 ~\mathrm dx \\
	&\quad + \left( \frac{1}{2d}\mu_{+}(I_k) - \frac{2\beta([d] + 1)}{d}\right) \int_{I_k} |f(x)|^2 ~\mathrm dx.
	\end{align*}
	Since $\mu_+(I_k) > 2/d$ is assumed, we have $1- \beta d([d] + 1) -\frac{d}{2}\mu_+(I_k) < 0$ and therefore Estimate \eqref{upperbound1} yields, using $1 - \beta d([d] + 1) > 0$,
	\begin{align*}
	a_{\mu}^{k}(f,f) &\geq \Bigg [  \left( 1 - \beta d([d] + 1) -\frac{d}{2}\mu_{+}(I_k) \right) \frac{\lambda + \frac{2\beta ([d] + 1)}{d}}{1 - \beta d([d]+1)} \\
	&\quad \quad +  \frac{1}{2d}\mu_{+}(I_k) - \frac{2\beta([d]+1)}{d}  \Bigg ] \int_{I_k} |f(x)|^2 ~\mathrm dx\\
	& \geq \lambda \int_{I_k} |f(x)|^2 ~\mathrm dx + \mu_{+}(I_k) \left( \frac{1}{2d} - \frac{d}{2} \cdot \frac{\lambda + \frac{2\beta([d] + 1)}{d}}{1-\beta d([d] + 1)} \right) \int_{I_k} |f(x)|^2 ~\mathrm dx.
	\end{align*}
    By using \eqref{lambda_d_relation}, we derive
    \begin{align*}
    \frac{1}{2d} - \frac{d}{2} \cdot \frac{\lambda + \frac{2\beta([d] + 1)}{d}}{1- \beta d([d] + 1)} \geq  \frac{1}{2d} - \frac{d}{2} \cdot \frac{\lambda + \frac{2\beta(d + 1)}{d}}{1-\beta d(d + 1)} = \frac{d^2\left(\frac{1}{d^2} - \frac{3 \beta}{d} - 3\beta -\lambda \right)}{2d(1-\beta d(d + 1))} = 0.
    \end{align*}
    Therefore, we conclude
    \begin{equation*}
    a_{\mu}^{k}(f,f) \geq \lambda \int_{I_k} |f(x)|^2 ~\mathrm dx
    \end{equation*}
    and consequently Glazman's lemma gives $N(\lambda, H_{\mu}^k) = 0$ whenever $\mu_+(I_k) > 2/d$. Thus, we obtain the estimate
	\begin{eqnarray*}
		N(\lambda,H_{\mu})\leq \sum_{\mu_+(I_k)\leq 2/d}N(\lambda,H_{\mu}^k).
	\end{eqnarray*}
	Let $\Delta_{N,k}$ be the Laplace operator on $I_k$ corresponding to Neumann boundary condition. Consider the operator 
	$$
	L_k = (1 - \beta d([d] + 1)) \Delta_{N,k} - \frac{2\beta ([d] + 1)}{d}.
	$$ 
	Its eigenvalues are given by
    \begin{equation*}
    \left\{(1 - \beta d([d] + 1))\left( \frac{2\pi n}{d} \right)^2 - \frac{2\beta ([d] + 1)}{d} \right\}_{n=0}^{\infty}.
    \end{equation*}
    Since, for $n \geq 1$,
    \begin{align*}
    (1 - \beta d([d] + 1))&\left( \frac{2\pi n}{d} \right)^2 - \frac{2\beta ([d] + 1)}{d}\\
    &\geq (1 - \beta d(d + 1))\left( \frac{2\pi n}{d} \right)^2 - \frac{2\beta(d + 1)}{d}\\
    &\geq (2\pi n)^2 \left(\frac{1}{d^2} - \frac{(2\pi n)^2 + 2}{(2\pi n)^2}\frac{\beta}{d}  - \frac{(2\pi n)^2 + 2}{(2\pi n)^2} \frac{\beta}{d} \right)\\
    &> \lambda,
    \end{align*}
    Glazman's lemma \ref{Glazman} gives $N(\lambda, L_k)\leq 1$. By \eqref{lower_bound_for form}, $H_{\mu}^k \geq L_k$ in the sense of quadratic forms. Hence,
    \begin{eqnarray}\label{upperbound3}
    N(\lambda,H_{\mu})\leq \sum_{\mu_+(I_k)\leq 2/d}N(\lambda,H_{\mu}^k)\leq \sum_{\mu_+(I_k)\leq 2/d}N(\lambda,L_k)\leq \sum_{\mu_+(I_k)\leq 2/d}1.
    \end{eqnarray}
    For $k \in \mathbb Z$ such that $\mu_+(I_k)\leq 2/d$, we denote $\tilde{I}_k=[(k-1)d+d/4, kd-d/4]$. Then, 
\begin{eqnarray*}
	\mu_+(\tilde{I}_k)\leq \mu_+(I_k) \leq 2/d.
\end{eqnarray*}
Therefore, we see that for any $x\in \tilde{I}_k$ we have the following inequality:
\begin{eqnarray*}
	q_{\mu_+}^*(x)\leq \left(\frac{2}{d}\right)^2.
\end{eqnarray*}
Thus,
\begin{eqnarray*}
	1=\frac{2}{d}\mathcal{L}(\tilde{I}_k)\leq \frac{2}{d}\mathcal{L}\left( \left\{x\in I_k: q_{\mu_+}^*(x)\leq \left(\frac{2}{d}\right)^2\right\} \right).
\end{eqnarray*}
Putting this into \eqref{upperbound3} gives
\begin{align*}
	N(\lambda,H_{\mu})&\leq \frac{2}{d}\sum_{k \in \mathbb Z} \mathcal{L} \left( \left\{x\in I_k: q_{\mu_+}^*(x)\leq\left(\frac{2}{d}\right)^2\right\} \right) \\
	& \leq \frac{2}{d} \mathcal{L}\left( \left\{x\in \mathbb{R}: q_{\mu_+}^*(x)\leq \left(\frac{2}{d}\right)^2\right\} \right).
\end{align*}
By recalling the definition of $d$, we complete the proof.
\end{proof}

\subsection{Lower bound for the spectral distribution function}
Next, we estimate the spectral distribution function from below.
\begin{theorem}\label{lower_bound}
	Let $\mu \in \mathfrak {M}_\beta$ and $\lambda \geq 0$. Then, the spectral distribution function of $H_\mu$ satisfies
	\begin{equation*}
	2\sqrt{\frac{3(\frac{\alpha}{l} + 2\lambda)}{16(\pi^2 + 1)(\alpha l+2)}} \mathcal{L}\left( \left\{x\in\mathbb{R}: \; q_{\mu_+}^{*}(x)< \frac{3(\frac{\alpha}{l} + 2\lambda)}{16(\pi^2 + 1)(\alpha l+2)}\right\}\right) \leq N(\lambda, H_\mu).
	\end{equation*}
\end{theorem}
\begin{proof}
	We first assume $\mu \in \mathfrak P$, in particular $\alpha = 0$. Obviously, if $\lambda = 0$, then the estimate holds. Let $\lambda>0$ and $I_k=((k-1)/\sqrt{\lambda}, k/\sqrt{\lambda}]$ for $k\in \mathbb{Z}$ and let $m_k = \frac{k - 1/2}{\sqrt{\lambda}}$ denote the midpoint of $I_k$. For $\varepsilon > 0$ let $\mu_\varepsilon$ denote the measure
	\begin{align*}
	\mu_\varepsilon(A) = \mu(A - \varepsilon).
	\end{align*}
	Recall that $\mu$ can have at most countably many point masses, i.e. there are at most countably many points $x \in \mathbb R$ such that $\mu(\{ x\}) \neq 0$. From this, it is an easy exercise to prove that there necessarily exists some small $\varepsilon > 0$ such that for all $k \in \mathbb Z$:
	\begin{equation}
	\mu_\varepsilon|_{I_k} \neq \delta_{m_k}.
	\end{equation}
	Since clearly $N(\lambda, H_\mu) = N(\lambda, H_{\mu_\varepsilon})$ for any $\varepsilon$, we might assume without loss of generality that we have for all $k \in \mathbb Z$:
	\begin{equation}\label{lower_bound_-1}
	\mu |_{I_k} \neq \delta_{m_k}
	\end{equation}
	We define the functions
	\begin{eqnarray*}
		\omega_k(x) := \begin{cases}
			1-\cos\left( 2\pi\sqrt{\lambda}\left(x-\frac{k}{\sqrt{\lambda}}\right)\right), ~&x \in I_k,\\
			0, & x \not \in I_k,
		\end{cases}
	\end{eqnarray*}
for $k\in \mathbb{Z}$. Note that $\operatorname{supp}(\omega_j)  \cap \operatorname{supp}(\omega_k) = \emptyset$ unless $k\in \{j-1,j,j+1\}$. Therefore, since $\omega_j = 0$ at the end points of $I_j$, we conclude that
\begin{eqnarray}\label{ort}
	a_\mu(\omega_j, \omega_k) = 0 = \langle \omega_j, \omega_k\rangle_{L^2}
	\quad \text{for } j\neq k.
\end{eqnarray}
For $k = j$ we obtain
	\begin{align}\label{lower_bownd_0}
\nonumber a_{\mu}(\omega_k,\omega_k)&= \int_{I_k} |\omega_k'(x)|^2~\mathrm dx+\int_{I_k}|\omega_k|^2 ~\mathrm d\mu\\
&\leq \left(2\pi\sqrt{\lambda}\right)^2\int_{I_k} \sin^2 \left( 2\pi\sqrt{\lambda} \left(x - \frac{k}{\sqrt{\lambda}}\right)\right) ~\mathrm dx +4\mu(I_k)\\
\nonumber&= 2\pi^2 \sqrt{\lambda} + 4\mu(I_k)
\end{align}
and
\begin{eqnarray*}
	\langle \omega_k, \omega_k\rangle_{L^2} = \frac{3}{2\sqrt{\lambda}}.
\end{eqnarray*}
Assume that there exists some $\xi \in I_k$ such that $\sqrt{\lambda}/2 > \sqrt{q_{\mu}^\ast(\xi)}$. Then, from the definition of the function $q_{\mu}^*$ we obtain
	\begin{equation}\label{lower_bownd_2}
\sqrt{\lambda}/2\geq \mu\left(\left[(\xi - 1/\sqrt{\lambda},\xi + 1/\sqrt{\lambda}\right]\right)\geq \mu(I_k),
\end{equation}
and hence
\begin{eqnarray}\label{lower_bound1}
	a_{\mu}(\omega_k,\omega_k)\leq 2\pi^2 \sqrt{\lambda} + 4\mu(I_k)\leq  2\sqrt{\lambda} (\pi^2 + 1) = \frac{4}{3}\left(\pi^2+1\right) \lambda \langle \omega_k, \omega_k\rangle_{L^2}.
\end{eqnarray}
Here, we used \eqref{lower_bownd_0} and \eqref{lower_bownd_2}. Let us note that at least one of these estimates is strict. Indeed, if $\mu(I_k) = 0$, then \eqref{lower_bownd_2} is strict, otherwise, \eqref{lower_bownd_0} is strict by \eqref{lower_bound_-1}. Therefore, the last estimate is strict, that is 
\begin{equation*}
a_{\mu}(\omega_k,\omega_k) < \frac{4}{3}\left(\pi^2+1\right) \lambda \langle \omega_k, \omega_k\rangle_{L^2},
\end{equation*}
whenever $I_k$ has a nonempty intersection with the set
\begin{equation*}
\left\{x\in\mathbb{R}: \; q_{\mu}^*(x)<\lambda/4\right\}.
\end{equation*}
Therefore, using the orthogonality relation \eqref{ort}, Lemma \ref{Glazman} gives
\begin{align*}
N\left(\frac{4}{3}\left(\pi^2+1\right)\lambda,H_\mu\right)&\geq \sum_{I_k\cap \left\{x\in\mathbb{R}: \; q_{\mu}^*(x)<\lambda/4\right\}\neq \emptyset} 1=\sqrt{\lambda}\sum_{I_k\cap \left\{x\in\mathbb{R}: \; q_{\mu}^*(x)<\lambda/4\right\}\neq \emptyset} \frac{1}{\sqrt{\lambda}}\\
&\geq\sqrt{\lambda}\mathcal{L}\left( \left\{x\in\mathbb{R}: \; q_{\mu}^*(x)<\lambda/4\right\} \right).
\end{align*}
Now, let $\mu \in \mathfrak M$ such that $\mu_- \neq 0$ with $l > 0$ and $\alpha \geq 0$. We let $I_k = (l(k-1), lk]$ for $k \in \mathbb Z$. Using the definition of the constants $l$ and $\alpha$ we obtain for $f \in D(a_\mu) = D(a_{\mu_+})$:
\begin{align*}
a_\mu(f,f) &= \int_{\mathbb R} |f'(x)|^2 ~\mathrm dx + \int_{\mathbb R} |f|^2 ~\mathrm d\mu_+ - \int_{\mathbb R} |f|^2 ~\mathrm d\mu_-\\
&= \int_{\mathbb R} |f'(x)|^2 ~\mathrm dx + \int_{\mathbb R} |f|^2 ~\mathrm d\mu_+ - \sum_{k \in \mathbb Z} \int_{I_k} |f|^2 ~\mathrm d\mu_-\\
&\leq \int_{\mathbb R} |f'(x)|^2 ~\mathrm dx + \int_{\mathbb R} |f|^2 ~\mathrm d\mu_+ - \alpha \sum_{k \in \mathbb Z} \inf_{x \in I_k} |f(x)|^2.
\end{align*}
Lemma \ref{lemma_Brinck} now yields:
\begin{align*}
a_\mu&(f,f) \\
&\leq \int_{\mathbb R} |f'(x)|^2 ~\mathrm dx + \int_{\mathbb R} |f|^2 ~\mathrm d\mu_+ + \alpha\sum_{k \in \mathbb Z} \left( \frac{l}{2} \int_{I_k} |f'(x)|^2 ~\mathrm dx - \frac{1}{2l} \int_{I_k} |f(x)|^2 ~\mathrm dx \right)\\
&= \left( 1 + \frac{\alpha l}{2} \right) \int_{\mathbb R} |f'(x)|^2 ~\mathrm dx + \int_{\mathbb R} |f|^2 ~\mathrm d\mu_+ - \frac{\alpha}{2l} \int_{\mathbb R} |f(x)|^2 ~\mathrm dx.
\end{align*}
This yields, in the sense of quadratic forms,
\begin{equation*}
a_\mu \leq \left( 1 + \frac{\alpha l}{2}\right) a_{\mu_+} - \frac{\alpha}{2l}.
\end{equation*}
Therefore, Glazman's Lemma gives
\begin{align*}
N(\lambda, H_\mu) \geq N\left( \lambda, \left( 1 + \frac{\alpha}{2l} \right) H_{\mu_+} - \frac{\alpha}{2l} \right) = N\left( \frac{\lambda + \frac{\alpha}{2l}}{1 + \frac{\alpha l}{2}}, H_{\mu_+} \right).
\end{align*}
Applying now the estimate from the positive case yields the result.
\end{proof}

\section{Applications}
In this section, we will give several applications of Theorems \ref{theorem_upperbound} and \ref{lower_bound}. Namely, we derive two criteria for the discreteness of the spectrum of operator $H_{\mu}$. Next, we obtain a necessary and sufficient conditions for the membership of the resolvents of $H_{\mu}$ to the Schatten classes $\mathfrak S_p$. 
\subsection{Discreteness of the spectrum}

First, we prove a criterion for the discreteness of the spectrum of $H_{\mu}$ in terms of the function $q^*_{\mu_+}$:

\begin{theorem}\label{theorem_discrete_spec_q}
	Let $\mu \in \mathfrak {M}_{\beta}$. The spectrum of $H_\mu$ is purely discrete if and only if
	\begin{eqnarray}\label{disc1}
	q_{\mu_+}^\ast(x) \to \infty \quad\text{as } |x| \to \infty.
	\end{eqnarray}
\end{theorem}
\begin{proof}
	Assume that \eqref{disc1} holds. Then, Theorem \ref{theorem_upperbound} implies that $N(\lambda,H_{\mu})<\infty$ for each $\lambda > 0$ so that the spectrum is discrete.
	
	Conversely, assume that $q_{\mu_+}^*(x)\nrightarrow\infty$ as $|x|\rightarrow+\infty$, then there exists a sequence $\{x_j\}_{j=1}^{\infty}$ and constant $\kappa >0$ such that $|x_j|\rightarrow\infty$ and $q_{\mu_+}^*(x_j)<\kappa$. By Lemma \ref{est_q*}, we obtain
	\begin{eqnarray*}
		q_{\mu_+}^*(x)<4\kappa \quad \text{for }
		\quad x\in \bigcup\limits_{i=1}^{\infty}\left[x_j-1/(1\sqrt{\kappa}), x_j+1/(1\sqrt{\kappa})\right].
	\end{eqnarray*}
	Now, Theorem \ref{lower_bound} with $\alpha = 0$ implies
	\begin{align*}
		N&\left(\frac{4}{3} \kappa (\pi^2 + 1), H_{\mu}\right)\\
		& \quad \geq \sqrt{\kappa}\mathcal{L} \left( \left\{x \in \mathbb R:~ q_{\mu_+}^*(x)\leq 4\kappa\right\} \right) \\
		& \quad \geq \sqrt{\kappa}\mathcal{L} \left( \left\{\bigcup\limits_{i=1}^{\infty}\left[x_j-1/(1\sqrt{\kappa}), x_j+1/(1\sqrt{\kappa})\right]\right\} \right) \\
		&\quad =\infty.
	\end{align*}
	Hence, $H_\mu$ cannot have discrete spectrum.
\end{proof}
Next we derive, as a corollary, Molchanov's discreteness criterion for the spectrum of $H_{\mu}$:
\begin{corollary}\label{corollary_discrete_spec_mu}
	Let $\mu \in \mathfrak {M}_{\beta}$. The operator $H_\mu$ has purely discrete spectrum if and only if for all $d> 0$ it holds 
	\begin{eqnarray}\label{disc2}
	\mu([x, x+d]) \to \infty \quad \text{as } |x| \to \infty.
	\end{eqnarray}
\end{corollary}
\begin{proof}
	We will first show that \eqref{disc1} is equivalent to
	\begin{equation}\label{disc3}
	\mu_+([x, x+d]) \to \infty \quad \text{as } |x| \to \infty
	\end{equation}
	for every $d > 0$. Assume that \eqref{disc1} does not hold, so that there exists a sequence $\{x_j\}_{j=1}^{\infty}$ and $\kappa>0$ such that $|x_j|\rightarrow\infty$ and $q_{\mu_+}^*(x_j)<\kappa$. Then, 
	$$
	\mu_+([x_n-1/\sqrt{\kappa}, x_n+1/\sqrt{\kappa}])\leq \sqrt{q_{\mu_+}^*(x_n)}\leq \sqrt{\kappa},
	$$
	which contradicts \eqref{disc3}.
	
	Conversely, suppose that \eqref{disc3} does not hold. Therefore, there exist $\kappa>0$, $d>0$, and a sequence $\{x_j\}_{j=1}^{\infty}$ such that $|x_j|\rightarrow\infty$ and $\mu_+([x_n-d/2,x_n+d/2])<\kappa$. Thus, there are two possibilities: 
	\begin{eqnarray*}
		\sqrt{q_{\mu_+}^*(x_n)}\leq d \quad or \quad \sqrt{q_{\mu_+}^*(x_n)}\leq \mu_+([x_n-d/2,x_n+d/2])<\kappa.	
	\end{eqnarray*}
	Therefore, \eqref{disc1} does not hold.
	Finally, since $\mu \in \mathfrak {M}_{\beta}$, it is simple to verify that \eqref{disc2} is indeed equivalent to \eqref{disc3}.
\end{proof} 
	\begin{remark}\label{disc cri}
		Originally, this result has been proven in \cite{Molchanov} for potentials being function bounded from below. The version stated here, Corollary \ref{corollary_discrete_spec_mu}, was proved in \cite{Albeverio_Kostenko_Malamud}. Later, it was improved in \cite{M}, namely they considered the potential satisfying Brinck's condition. We also note that sufficient conditions for discreteness of non-semibounded Hamiltonians $H_\mu$ with discrete measure $\mu$ were obtained in \cite{KM2010}, and \cite{IK2010} (see also the survey \cite{KM2013}).
	\end{remark}

\subsection{Schatten class membership of resolvents}
\begin{definition}
	For $p>0$, we say that a compact operator $A$ belongs to the Schatten-von Neumann class $\mathfrak S_p$ if 
	\begin{equation*}
	\sum_{k=1}^{\infty} s_k(A)^p <\infty,
	\end{equation*}
	where $s_k(A)$ are the eigenvalues of the operator $\sqrt{A^*A}$.
\end{definition}
\begin{theorem}\label{Schatten}
	Let $\beta \geq 0$, $\mu \in \mathfrak {M}_{\beta}$, and $p>1/2$. Further, assume that the spectrum of $H_{\mu}$ is discrete. Then, the resolvents of $H_{\mu}$ belong to the $p$-Schatten class $\mathfrak S_p$ if and only if $1/q_{\mu_+}^* \in L^{p-1/2}(\mathbb{R})$.
    Moreover, if $\mu \in \mathfrak P$, for the eigenvalues $\{\lambda_k\}_{k=1}^{\infty}$ of $H_{\mu}$, we obtain
    \begin{eqnarray}\label{S_norm}
   	\frac{2p}{p-1/2} \left(\frac{3}{32(\pi^2+1)}\right)^p \|1/q_{\mu_+}^*\|_{L^{p-1/2}}^{p-1/2} \leq \sum_{ k = 1}^{\infty} \frac{1}{\lambda_k^p} \leq \frac{p}{p-1/2} 5^p \|1/q_{\mu_+}^*\|_{L^{p-1/2}}^{p-1/2}.
   \end{eqnarray}
\end{theorem}

\begin{proof}
	Since the Schatten-von Neumann classes are ideals, by the resolvent identity it suffices to prove that one particular resolvent is contained in $\mathfrak S_p$. Let $\{\lambda_j\}_{j=1}^{\infty}$ be the non-decreasing sequence of eigenvalues of $H_{\mu}$. Set $q_0^\ast := \min_{x \in \mathbb R} q_{\mu_+}^\ast(x)$, which is strictly positive if $H_\mu$ has discrete spectrum by Theorem \ref{theorem_discrete_spec_q}. As we will show later in Theorem \ref{estimate_smallest_eigenvalue}, the smallest eigenvalue of $H_\mu$ is strictly larger than $-c_0 := -\frac{3\beta }{2}(2 + \sqrt{q_0^\ast})$. Therefore, convergence of the following integrals at $-c_0$ is never an issue (since $N(\lambda, H_\mu) = 0$ for $\lambda$ close to $-c_0$). For $\varepsilon>0$, we estimate
	\begin{align}
		\int_{-c_0}^{\lambda_k-\varepsilon} \frac{\mathrm d N(\lambda, H_{\mu})}{(\lambda + c_0)^p} &\leq \sum_{ j = 1}^{k} \frac{1}{(\lambda_j+c_0)^p} \leq \int_{-c_0}^{\lambda_k + \varepsilon} \frac{\mathrm d N(\lambda, H_{\mu})}{(\lambda+c_0)^p}.
	\end{align}
Integrating by parts gives 
\begin{align*}
	\sum_{ j = 1}^{k} \frac{1}{(\lambda_j+c_0)^p}	 &\leq \frac{N(\lambda_k  + \varepsilon, H_{\mu})}{(\lambda_k + c_0 + \varepsilon)^p} + p \int_{-c_0}^{\lambda_k + \varepsilon} \frac{N(\lambda, H_{\mu})}{(\lambda + c_0)^{p+1}} ~\mathrm d\lambda\\
	& \leq p N(\lambda_k + \varepsilon, H_{\mu}) \int_{\lambda_k + \varepsilon}^{\infty} \frac{\mathrm d\lambda}{(\lambda+c_0)^{p + 1}} + p \int_{-c_0}^{\lambda_k + \varepsilon} \frac{N(\lambda, H_{\mu})}{(\lambda+c_0)^{p+1}} ~\mathrm d\lambda\\
	& \leq p \int_{-c_0}^{\infty} \frac{N(\lambda, H_{\mu})}{(\lambda+c_0)^{p+1}} ~\mathrm d\lambda.
\end{align*}
For the lower estimate, we obtain through integration by parts
\begin{align*}
	\sum_{j = 1}^{k} \frac{1}{(\lambda_j+c_0)^p} &\geq \frac{N(\lambda_k - \varepsilon, H_{\mu,})}{(\lambda_k +c_0- \varepsilon)^p} + p \int_{-c_0}^{\lambda_k - \varepsilon} \frac{N(\lambda, H_{\mu})}{(\lambda+c_0)^{p+1}} ~\mathrm d\lambda\\
	&\geq p \int_{-c_0}^{\lambda_k - \varepsilon} \frac{N(\lambda, H_\mu)}{(\lambda + c_0)^{p+1}} ~\mathrm d\lambda.
\end{align*}
Comparing the upper and lower bound and letting $k \to \infty$, we obtain
\begin{align*}
	\sum_{ j = 1}^{\infty} \frac{1}{(\lambda_j+c_0)^p} = p \int_{-c_0}^{\infty} \frac{N(\lambda, H_{\mu})}{(\lambda+c_0)^{p+1}} ~\mathrm d\lambda.
\end{align*}
By Theorem \ref{theorem_upperbound}, there is $\Lambda \geq 0$ such that for $\lambda > \Lambda$ we have $N(\lambda-c_0, H_\mu) \leq M(5\lambda)$. Thus:
\begin{align*}
\int_{-c_0}^\infty \frac{N(\lambda, H_{\mu})}{(\lambda+c_0)^{p+1}} ~\mathrm d\lambda &\leq \int_{-c_0}^{\Lambda-c_0} \frac{N(\lambda, H_\mu)}{(\lambda+c_0)^{p+1}}~\mathrm d\lambda + \int_{\Lambda -c_0}^\infty \frac{N(\lambda, H_\mu)}{(\lambda+c_0)^{p+1}} ~\mathrm d\lambda\\
&=  \int_{-c_0}^{\Lambda-c_0} \frac{N(\lambda, H_\mu)}{(\lambda+c_0)^{p+1}}~\mathrm d\lambda + \int_{\Lambda}^\infty \frac{N(\lambda-c_0, H_\mu)}{\lambda^{p+1}} ~\mathrm d\lambda\\
&\leq  \int_{-c_0}^{\Lambda-c_0} \frac{N(\lambda, H_\mu)}{(\lambda+c_0)^{p+1}}~\mathrm d\lambda + \int_{\Lambda}^\infty \frac{M(5\lambda)}{\lambda^{p+1}} ~\mathrm d\lambda\\
&\leq  \int_{-c_0}^{\Lambda-c_0} \frac{N(\lambda, H_\mu)}{(\lambda+c_0)^{p+1}}~\mathrm d\lambda + \int_{0}^\infty \frac{M(5\lambda)}{\lambda^{p+1}} ~\mathrm d\lambda\\
&= \int_{-c_0}^{\Lambda-c_0} \frac{N(\lambda, H_{\mu})}{(\lambda+c_0)^{p+1}} ~\mathrm d\lambda + 5^p \int_0^\infty \frac{M(\lambda)}{\lambda^{p+1}}~\mathrm d\lambda.
\end{align*}	
We have, using Fubini's theorem,
\begin{align*}
	\int_{0}^{\infty} \frac{ M(\lambda)}{\lambda^{p+1}} ~\mathrm d\lambda &= \int_{0}^{\infty} \frac{ \sqrt{\lambda}\mathcal{L}( \{x\in \mathbb{R}: \; q^*_{\mu_+}(x)\leq \lambda\})}{\lambda^{p+1}} ~\mathrm d\lambda \\
	 &=  \int_{-\infty}^{\infty} \int_{q_{\mu_+}^*(x)}^{\infty} \frac{1}{\lambda^{p + \frac{1}{2}}} ~\mathrm d\lambda ~\mathrm dx\\ 
	 &= \frac{1}{p-1/2} \int_{-\infty}^{\infty} (q_{\mu_+}^*(x))^{1/2 - p}~\mathrm dx.
\end{align*}

By Theorem \ref{lower_bound}, there is $\Lambda' \geq 0$ such that for all $\lambda > \Lambda'$ we have $N(\lambda-c_0, H_\mu) \geq 2M(\frac{3\lambda}{16(\pi^2 + 1)(\alpha l+2)})$. Similarly as above, one then shows
\begin{align*}
\int_{-c_0}^{\infty}& \frac{N(\lambda, H_\mu)}{(\lambda + c_0)^{p+1}} ~\mathrm d\lambda \\
&\geq \int_{-c_0}^{\Lambda'-c_0} \frac{N(\lambda, H_\mu)}{(\lambda + c_0)^{p+1}}~\mathrm d\lambda + \int_{\Lambda'}^\infty \frac{2M\left(3\lambda/\left(16(\pi^2 + 1)(\alpha l+2)\right)\right)}{\lambda^{p+1}}~\mathrm d\lambda.
\end{align*}
The second integral can be computed as
\begin{align*}
\int_{\Lambda}^\infty &\frac{2M\left(3\lambda/\left(16(\pi^2 + 1)(\alpha l+2)\right)\right)}{\lambda^{p+1}}~\mathrm d\lambda\\
 &= 2\left( \frac{3}{16(\pi^2 + 1)(\alpha l+2)} \right)^{p} \int_{\Lambda''}^\infty \frac{M(\lambda)}{\lambda^{p+1}} ~\mathrm d\lambda\\
&= 2\left( \frac{3}{16(\pi^2 + 1)(\alpha l + 2)} \right)^{p} \frac{1}{p-1/2} \int_{-\infty}^\infty \max\{ q_{\mu_+}^\ast(x), \Lambda''\}^{1/2-p}~\mathrm dx,
\end{align*}
where $\Lambda'' = \frac{3\Lambda}{16(\pi^2 + 1)(\alpha l + 2)}$. Therefore, we obtain $\int_{-c_0}^\infty \frac{N(\lambda, H_\mu)}{(\lambda + c_0)^{p+1}} ~d\lambda < \infty$ if and only if $1/q_{\mu_+}^\ast \in L^{p - 1/2}(\mathbb R)$. For the explicit estimates in the case $\mu \in \mathfrak P$, observe that we can simply set $c_0 = \Lambda = \Lambda' = \Lambda'' = 0$ in all the above computations, which yields the formula. 
\end{proof}

	\begin{remark}
		It seems to be a more common theme to investigate Schatten class properties of resolvent differences, instead of Schatten class properties of the resolvents on their own. Of the numerous papers in this direction, we only want to mention the work \cite{Ananieva} for such results on Sturm-Liouville operators with discrete measure potentials.
	\end{remark}

\subsection{Estimates for the eigenvalues}
In this subsection, we will obtain several estimates for the eigenvalues of $H_{\mu}$. We begin by estimating the number of negative eigenvalues:
\begin{corollary}
	For $\mu \in \mathfrak M_\beta$ with $\beta > 0$, the number of negative eigenvalues of $H_\mu$ is bounded by 
	\begin{equation*}
	N(0,H_{\mu}) \leq(\sqrt{9\beta^2 + 12\beta} + 3\beta) \mathcal{L} \left( \left \{ x\in \mathbb R: ~q_{\mu_+}^\ast(x) \leq (\sqrt{9\beta^2 + 12\beta} + 3\beta)^2 \right \} \right).
	\end{equation*}
	In particular, if $q^*_{\mu_+} > (\sqrt{9\beta^2 + 12\beta} + 3\beta)^2$, then $\sigma(H_{\mu}) \cap (-\infty, 0] = \emptyset$.
\end{corollary}
\begin{proof}
	Follows from Theorem \ref{theorem_twosided_estimate} with $\lambda=0$.
\end{proof}
Next we obtain a two-sided inequality for the lower bound of the spectrum.
\begin{theorem}\label{estimate_smallest_eigenvalue}
	Let $\mu \in \mathfrak {M}_{\beta}$, $\alpha, l \geq 0$ be the constants attributed to $\mu_-$ and $\lambda_1$ be the lower bound of the spectrum of $H_{\mu}$. Then
	\begin{equation}\label{estimate_smallest_eigenvalue0}
	\frac{1}{4}q_0^* - \frac{3\beta}{2}(2 + \sqrt{q_0^\ast}) \leq \lambda_1 \leq \frac{8}{3}(\pi^2 + 1)(\alpha l+2)q_0^\ast - \frac{\alpha}{2l},
	\end{equation}
	where $q_0^*=\inf_{x\in \mathbb{R}} q_{\mu_+}^*(x)$.
\end{theorem}
\begin{proof}
	Let $\lambda>\frac{8}{3}(\pi^2 + 1)(\alpha l+2)q_0^\ast - \frac{\alpha }{2l}$ which implies 
	\begin{equation*}
	\frac{3(\frac{\alpha}{l}+2\lambda)}{16(\pi^2 + 1)(\alpha l+2)} > q_0^\ast.
	\end{equation*}
	Then, Theorem \ref{theorem_twosided_estimate} gives
	\begin{align*}
	N&(\lambda,H_{\mu})\\ 
	&\geq 2\sqrt{\frac{3(\frac{\alpha}{l} + 2\lambda)}{16(\pi^2+1)(\alpha l+2)}} \mathcal{L}\left( \left\{x\in\mathbb{R}: \; q_{\mu_+}^{*}(x)\leq \frac{3(\frac{\alpha}{l}+2\lambda)}{16(\pi^2+1)(\alpha l+2)}\right\}\right) \\
	&>0,
	\end{align*}
since $q_{\mu_+}^\ast$ is continuous. Therefore, $\lambda_1<\lambda$, which yields $\lambda_1\leq \frac{8}{3}(\pi^2 + 1)(\alpha l+2)q_0^\ast - \frac{\alpha}{2l}$.
	
	Next, assume that $\lambda< \frac{1}{4}q_0^* - \frac{3\beta}{2}(2 + \sqrt{q_0^\ast})$, which implies
	\begin{equation*}
q_0^\ast > (\sqrt{4\lambda + 9\beta^2 + 12\beta} + 3\beta)^2.	
	\end{equation*}
	Then,  Theorem \ref{theorem_twosided_estimate} gives
	\begin{align*}
	N&(\lambda, H_{\mu})\\ &\leq (\sqrt{4\lambda + 9\beta^2 + 12\beta} + 3\beta) \mathcal{L}\left( \{x\in \mathbb{R}: \; q_{\mu_+}^{*}(x)\leq (\sqrt{4\lambda + 9\beta^2 + 12\beta} + 3\beta)^2 \}\right)\\
	&\leq (\sqrt{4\lambda + 9\beta^2 + 12\beta} + 3\beta) \mathcal{L} \left( \{ x \in \mathbb R: ~q_{\mu_+}^\ast(x) < q_0^\ast \} \right)\\
	&= 0.
	\end{align*}
	Therefore $\lambda_1\geq \lambda$, hence $\frac{1}{4}q_0^* - \frac{3\beta}{2}(2 + \sqrt{q_0^\ast}) \geq \lambda_1$.
\end{proof}

This gives us the following result.

\begin{corollary}
	Let $\beta \geq 0$, $\mu \in \mathfrak {M}_{\beta}$, and $\lambda_1$ be the lower bound of the spectrum of $H_{\mu}$. Then $\lambda_1 \geq -3\beta$.
\end{corollary}

\begin{proof}
	For $\varepsilon > 0$, consider the operators $H_{-\mu_-}$ and $H_{-\mu_- + \varepsilon \mathcal{L}}$, where $\mathcal{L}$ is the Lebesgue measure. Let $\lambda_1(H_{-\mu_-})$, $\lambda_1(H_{-\mu_- + \varepsilon \mathcal{L}})$ be the lower bound for the spectrum of $H_{-\mu_-}$, $H_{-\mu_- + \varepsilon \mathcal{L}}$, respectively. Since $q^*_{\varepsilon \mathcal L} = \varepsilon$, Theorem \ref{estimate_smallest_eigenvalue} implies
	\begin{equation*}
	\lambda_1(H_{-\mu_-}) + \varepsilon = \lambda_1(H_{-\mu_- + \varepsilon \mathcal{L}})\geq \frac{1}{4} \varepsilon - 3\beta -\frac{3\beta}{2}\sqrt{\varepsilon}.
	\end{equation*}
	Since $\mathrm{D}(a_\mu) \subset \mathrm{D}(a_{-\mu_{-}})$ and $a_{\mu}\geq a_{-\mu_{-}}$ on $\mathrm{D}(a_\mu)$, Glazman's Lemma \ref{Glazman} implies that $\lambda_1 \geq \lambda_1(H_{-\mu_-})$, and therefore,
	\begin{equation*}
	\lambda_1 \geq - 3\beta - \frac{3}{4} \varepsilon - \frac{3\beta}{2}\sqrt{\varepsilon}
	\end{equation*}
	for any $\varepsilon > 0$. This complete the proof.
\end{proof}

Next, we estimate the lower bound of the essential spectrum.
\begin{theorem}\label{est_essential_spectrum}
	Let $\mu \in \mathfrak {M}_{\beta}$ and $\Lambda :=\inf \sigma_{ess} (H_\mu)$. Then,
	\begin{equation}\label{est_essential_spectrum0}
	\frac{1}{4}Q_{\mu_+} - \frac{3\beta}{2}(2 + \sqrt{Q_{\mu_+}}) \leq \Lambda \leq \frac{8}{3}(\pi^2 + 1)(\alpha l+2)Q_{\mu_+} - \frac{\alpha}{2l} ,
	\end{equation}
	where
	\begin{equation*}
	Q_{\mu_+} := \liminf_{x \to \pm \infty} q_{\mu_+}^\ast(x).
	\end{equation*}
\end{theorem}
\begin{remark}
	From Theorems \ref{estimate_smallest_eigenvalue} and \ref{est_essential_spectrum}, we see that if the right-hand side of \eqref{estimate_smallest_eigenvalue0} is less than the left-hand side of \eqref{est_essential_spectrum0}, then there exists an eigenvalue below the essential spectrum.
\end{remark}
\begin{proof}[Proof of Theorem \ref{est_essential_spectrum}]
The estimates are proven similarly to those in Theorem \ref{estimate_smallest_eigenvalue}. 
	Let $\lambda< \frac{1}{4} Q_{\mu_+} - \frac{3\beta}{2}(2 +\sqrt{Q_{\mu_+}})$. Then, for $\varepsilon > 0$ sufficiently small we still have
	\begin{equation*}
	\lambda + \varepsilon < \frac{1}{4} Q_{\mu_+} - \frac{3\beta}{2}(2 +\sqrt{Q_{\mu_+}}),
	\end{equation*}
	i.e.
	
	\begin{equation*}
	(\sqrt{4(\lambda + \varepsilon) + 9\beta^2 + 12\beta} + 3\beta)^2 < Q_{\mu_+}.
	\end{equation*}
	This implies
	\begin{equation*}
	\mathcal{L}\{x\in \mathbb{R}: \; q^*_{\mu_+}(x)\leq(\sqrt{4(\lambda + \varepsilon) + 9\beta^2 + 12\beta} + 3\beta)^2\}<\infty,
	\end{equation*}
	which in turn yields, by Theorem \ref{theorem_twosided_estimate}, $N(\lambda + \varepsilon,H_{\mu})<\infty$, and therefore, $\lambda < \Lambda$, i.e. $\frac{1}{4} Q_{\mu_+} - \frac{3\beta}{2}(2 +\sqrt{Q_{\mu_+}}) \leq \Lambda$.
	The other inequality now follows by imitating the steps from the proof of Theorem \ref{estimate_smallest_eigenvalue} analogously.
\end{proof}

For the remaining part of this section, let $\mu \in \mathfrak P$ be such that $H_\mu$ has discrete spectrum. Recall that the function $M(\lambda)$ is defined as
\begin{equation*}
M(\lambda) = \sqrt{\lambda} \mathcal{L} ~ \left( \left \{ x \in \mathbb R: ~ q_{\mu}^\ast(x) < \lambda \right \} \right).
\end{equation*}
By Theorem \ref{theorem_twosided_estimate}, $M(\lambda)$ is finite for every $\lambda \in [0, \infty)$. Since $q_{\mu}^\ast$ is continuous and is strictly positive, $M: [0, \infty) \to [0, \infty)$ is increasing and satisfies $M(0) = 0$, $\lim_{\lambda \to \infty} M(\lambda) = \infty$. Let $\xi = \sup\{\lambda>0: \; M(\lambda) < 1\}$. We define
\begin{equation*}
\widetilde{M}(\lambda)=
\begin{cases}
M(\lambda) & \text{for }\lambda \geq \xi,\\
\frac{\lambda}{\xi} & \text{for }\lambda<\xi.
\end{cases}
\end{equation*}
Note that $\widetilde{M}$ is a strictly increasing function on $[0, \infty)$, so that it has an inverse function which we denote by $F$. Next, we estimate the eigenvalues of $H_\mu$ in terms of the function $F$.

\begin{theorem}
	Let $\mu \in \mathfrak P$ be such that $H_{\mu}$ has discrete spectrum and its eigenvalues $\{\lambda_k\}_{k=1}^{\infty}$ have multiplicities equal to 1. Then, for every $n\in \mathbb{N}$:
	\begin{equation*}
	\frac{1}{4}F(n)\leq \lambda_n\leq \frac{16(\pi^2+1)}{3}F(n).
	\end{equation*}
\end{theorem}

\begin{proof}
	By Theorem \ref{theorem_twosided_estimate}, we know
	\begin{equation*}
	M\left(\frac{3\lambda}{16(\pi^2+1)}\right)\leq N(\lambda, H_\mu)\leq M(4\lambda).
	\end{equation*}
	Let $\varepsilon>0$ such that $\lambda=\lambda_n + \varepsilon < \lambda_{n+1}$, then
	\begin{equation*}
	M\left(\frac{3(\lambda_n + \varepsilon)}{16(\pi^2+1)}\right)\leq n\leq M(4(\lambda_n+ \varepsilon)).
	\end{equation*}
	If $\frac{3(\lambda_n + \varepsilon)}{16(\pi^2+1)} < \xi$, where $\xi = \sup\{\lambda>0: \; M(\lambda) < 1\}$, then $\widetilde{M}\left(\frac{3(\lambda_n + \varepsilon)}{16(\pi^2+1)}\right) \leq 1 \leq n$. Since $M(4(\lambda_n+ \varepsilon)) \geq n \geq 1$, we know that $M(4(\lambda_n+ \varepsilon)) = \widetilde{M}(4(\lambda_n+ \varepsilon))$. Therefore, we conclude
	\begin{equation*}
	\widetilde{M}\left(\frac{3(\lambda_n + \varepsilon)}{16(\pi^2+1)}\right)\leq n\leq \widetilde{M}(4(\lambda_n+ \varepsilon)).
	\end{equation*}
	Applying the inverse function $F$, which is also increasing, gives
	\begin{equation*}
	\frac{3(\lambda_n + \varepsilon)}{16(\pi^2+1)}\leq F(n)\leq 4(\lambda_n+ \varepsilon)),
	\end{equation*}
	for sufficiently small $\varepsilon > 0$. This complete the proof.
\end{proof}

\begin{corollary}
	Assume that $\mu \in \mathfrak P$ and $H_{\mu}$ has discrete spectrum. Let $\nu$ be an eigenvalue of $H_{\mu}$, then
	\begin{equation*}
	\dim Eig(H_{\mu}, \nu)\leq M(4\nu)-2M\left(\frac{3\nu}{16(\pi^2+1)}\right).
	\end{equation*}
\end{corollary}
\begin{proof}
	Theorem \ref{theorem_twosided_estimate} implies
	\begin{align*}
	2M\left(\frac{3(\nu-\varepsilon)}{16(\pi^2+1)}\right)&\leq N(\nu-\varepsilon,H_{\mu})\\
	&=N(\nu,H_{\mu})-\dim Eig(H_{\mu}, \nu)\\
	&\leq M(4\nu) -\dim Eig(H_{\mu}, \nu),
	\end{align*}
	for sufficiently small $\varepsilon>0$. This completes the proof.
\end{proof}

\section{Examples}

In this section, we will consider several examples of Sturm-Liouville operators with measure potentials and derive their spectral properties.

\begin{example}\label{exmample1}
	Let $f: \mathbb{Z} \rightarrow \mathbb{Z}$ be a non-negative function. Consider the measure $\mu = \sum_{k\in \mathbb{Z}} \delta_{t_k},$ where $\{t_k\}_{k\in \mathbb{Z}}$ is the sequence such that each $[m,m+1] \cap \{t_k\}_{k\in \mathbb{Z}}$ contains $f(m)$ points, which are uniformly distributed in $(m,m+1)$. In other words, the function $f$ determines the density of $\{t_k\}_{k\in \mathbb{Z}}$ on $\mathbb{R}$. Then
	
	\begin{itemize}
		\item[(i)]  The spectrum of $H_{\mu}$ is discrete if and only if $f(k)\rightarrow \infty$ as $k\rightarrow \infty$;
		\item[(ii)] Assume that the spectrum is discrete.  Let $p > 1/2$, then $H_{\mu}^{-1} \in \mathfrak{S}_p$ if and only if $\left\{\frac{1}{f(k)}\right\}_{k\in \mathbb{Z}} \in l^{p-1/2}(\mathbb{Z})$.
	\end{itemize}
    The first statment follows from Corollary \ref{corollary_discrete_spec_mu}. Let us check the second claim. Let $d>0$, one can check that, for sufficiently large $|x|$,
    \begin{equation*}
    \frac{d}{2} (f[x] + 1) \leq \mu \left(x-\frac{d}{2}, x+\frac{d}{2}\right) \leq d( f([x] - 1) +f([x] - 1) +f([x] - 1) + 3) + 3.
    \end{equation*}
    By definition of $q^*_{\mu}$, we conclude 
    \begin{equation*}
    \mu\left(x - \frac{1}{2\sqrt{q_{\mu}^*(x)+\varepsilon}},x - \frac{1}{2\sqrt{q_{\mu}^*(x)+\varepsilon}}\right)   \leq \sqrt{q^*_{\mu}(x) + \varepsilon},
    \end{equation*}
    \begin{equation*}
     \sqrt{q^*_{\mu}(x) -\varepsilon}  \leq \mu\left(x - \frac{1}{2\sqrt{q_{\mu}^*(x)-\varepsilon}},x - \frac{1}{2\sqrt{q_{\mu^*(x)-\varepsilon}}}\right).
    \end{equation*}
    Therefore, the exist $C_1$, $C_2>0$ such that 
    \begin{equation*}
    C_1(f[x] + 1) \leq q_{\mu}^*(x) \leq C_2 ( f([x] - 1) +f([x] - 1) +f([x] - 1)),
    \end{equation*}
    and hence, Theorem \ref{Schatten} implies (ii).
\end{example}

\begin{example}
Let $a>0$ and $\mu_+ = \frac{1}{a} \sum_{k\in \mathbb{Z}} \delta_{t_k}$, $\mu_- = \sum_{k\in \mathbb{Z}} \delta_{k}$, where $\{t_k\}_{k\in \mathbb{Z}}$ is a sequence such that  $t_k \to \pm \infty$ as $k \to \pm \infty$, $|t_k - t_{k-1}| \rightarrow 0$ as $k\rightarrow \pm \infty$ and $\max_{k\in \mathbb{Z}}|t_k - t_{k-1}| = a$. Then
\begin{itemize}
	\item[(i)]  The spectrum of $H_{\mu}$ is discrete;
	\item[(ii)] If $a> 4(\pi^2 + 1)^{1/2}$, then $\lambda_1 < 0$, where $\lambda_1= \min \sigma(H_{\mu})$;
	\item[(iii)] If $a< -3 +\sqrt{11}$, then $\lambda_1 > 0$.
\end{itemize}
The first statement follows from Corollary \ref{corollary_discrete_spec_mu}. Since $\max_{k\in \mathbb{Z}}|t_k - t_{k-1}| = a$, $q^*_{0} = \frac{1}{a^2}$, and (ii), (iii) follow from Theorem \ref{estimate_smallest_eigenvalue}.
\end{example}

\begin{example}
	Let $c>0$ and $\mu = \sum_{k\in \mathbb{Z}} |k|c \delta_k.$ Then
	\begin{itemize}
		\item[(i)]  The spectrum of $H_{\mu}$ is not discrete;
		\item[(ii)] If $c \leq \frac{1}{448}\sqrt{\frac{3}{\pi^2 + 1}}$, then there exists an eigenvalue below the essential spectrum.
	\end{itemize}
Corollary \ref{corollary_discrete_spec_mu} implies (i). To show (ii), note that $\liminf_{y \to \pm \infty}q^*_{\mu_+}(y) =1 $, and since $\mu([-4\sqrt{\frac{\pi^2+1}{3}}, 4\sqrt{\frac{\pi^2+1}{3}}]) = 56c$, we obtain

\begin{equation*}
q^*_{\mu_+}(0) \leq \frac{3}{64(\pi^2 + 1)} = \frac{3}{64(\pi^2 + 1)}\liminf_{y \to \pm \infty}q^*_{\mu_+}(y).
\end{equation*}
Theorems \ref{estimate_smallest_eigenvalue} and \ref{est_essential_spectrum} give (ii).
\end{example}

\begin{example}
	Let $\mu = \sum_{k\in \mathbb{Z}} \delta_{t_k},$ where $t_0=0$, $t_{\pm k}=\pm\sum_{j=1}^{k}\frac{1}{j}$, then
	\begin{itemize}
		\item[(i)]  The spectrum of $H_{\mu}$ is discrete;
		\item[(ii)] The first eigenvalue, $\lambda_1$, satisfies $1/4\leq \lambda_1\leq \frac{4}{3}(\pi^2 + 1)$;
		\item[(iii)] For $p>1/2$, $H_{\mu}^{-1} \in \mathfrak{S}_p$.
	\end{itemize}
\end{example}

Next, we give an example with discontinuous potential.

\begin{example}
	Let $f:\mathbb{R}\rightarrow\mathbb{R}$ be a function such that $f(x) = c_k$ for $x\in (k,k+1]$, where $\{c_k\}_{k\in \mathbb{Z}}$ are non-negative numbers and $\mathrm d\mu(x) = f(x) ~\mathrm dx$. Then
	\begin{itemize}
		\item[i.]  The spectrum of $H_{\mu}$ is discrete if and only if $c_k\rightarrow \infty$ as $k\rightarrow \infty$;
		\item[ii.] Assume that the spectrum is discrete.  Let $p > 1/2$, then $H_{\mu}^{-1} \in \mathfrak{S}_p$ if and only if $\{\frac{1}{c_k}\}_{k\in \mathbb{Z}} \in l^{p-1/2}(\mathbb{Z})$.
	\end{itemize}
By the same arguments that we gave in Example \ref{exmample1}, one can verify (i), (ii).
\end{example}

\begin{example}
		Let $\mathrm d\mu(x) = |x|^\kappa~\mathrm dx$ for some $\kappa \geq 1$. By Lemma \ref{Lemma_convex_concave}, $1/q_\mu^\ast(x) \leq 1/f(x) = |x|^{-\kappa}$ for $|x|$ large. $(|x|^{-\kappa})^{p-1/2}$ is integrable at infinity for $p > 1/\kappa + 1/2$. Hence, $(H_\mu)^{-1} \in \mathfrak S^p$ for such $p$.
\end{example}

\begin{example}
		Let $\mathrm d\mu(x) = |x|^\kappa ~\mathrm dx$ for $\kappa \geq 1$ as above and $\nu$ any positive finite Radon measure. By Lemma \ref{perturation_lemma}, the resolvents of $H_{\mu + \nu}$ are also in $\mathfrak S^p$ for $p > 1/\kappa + 1/2$. In particular, if $\lambda_n$ is the eigenvalue sequence of $H_{\mu + \nu}$, then $(1/\lambda_n) \in O(1/n^{2c/(2+c)})$.
\end{example}

\begin{example}
		For $n \in \mathbb N$ let $\mu \in \mathfrak M_\beta$ be such that $\mathrm d\mu(x) = \ln(|x|)^n~\mathrm dx$ for $|x|$ sufficiently large. It is easy to verify that $H_\mu$ has discrete spectrum. Since $x \mapsto \ln(|x|)^n$ is concave for $|x|$ large enough, Lemma \ref{Lemma_convex_concave} implies that $1/q_{\mu_+}^\ast(x) \geq 1/\ln(|x|)^n$ for $|x|$ large. As $(1/\ln(|x|)^n)^{p-1/2}$ is not integrable at infinity for any $p > 1/2$, the resolvents of $H_\mu$ are compact but in no Schatten class. Again, perturbing $\mu$ by finite measures does not change these properties.
\end{example}

\parindent0cm
\setlength{\parskip}{\baselineskip}

\bibliographystyle{alpha}
\bibliography{FN_Bib}

\setlength{\parskip}{0pt}

\end{document}